%% file: stability.tex
\documentclass{amsart}

\input{decls}

\tikzset{lab/.style={auto,font=\scriptsize}} 
\usetikzlibrary{arrows}

\usepackage{xcolor}
\usepackage{fixme}
\fxsetup{
    status=draft,
    author=,
    layout=margin,
    theme=color
}

\definecolor{fxnote}{rgb}{1.0000,0.0000,0.0000}
\colorlet{fxnotebg}{yellow}

\newcommand{\D}{\sD}

\autodefs{\cDER\cCat\cGrp\cCAT\cMONCAT\cTriaCAT\cAddCAT\cPreCAT\cPtCAT\bSet\cProf\bCat
\cPro\cMONDER\cBICAT\ncomp\nid\ncoker\niso\ndia\nIm\sProf\ncyl\fC\fZ\fT\nhocolim\nholim\cPsNat\ncoll\nCh\bsSet\cAlg\cIm\cI\cP}

\let\oldboxtimes\boxtimes
\def\boxtimes{\mathrel{\oldboxtimes}}

\newcommand{\fib}{\mathsf{fib}}
\newcommand{\cof}{\mathsf{cof}}

\def\ccsub{_{\mathrm{cc}}}
\def\pdh(#1,#2){\llbracket #1,#2\rrbracket}
\def\ldh(#1,#2){\llbracket #1,#2\rrbracket\ccsub}
\def\pend(#1){\pdh(#1,#1)}
\def\lend(#1){\ldh(#1,#1)}

\def\DTl#1#2#3#4#5#6#7{%
  \xymatrix@C=3pc{{#1} \ar[r]^-{#2} &
    {#3} \ar[r]^-{#4} &
    {#5} \ar[r]^-{#6} &
    {#7}
  }}

\newsavebox{\tvabox}
\savebox\tvabox{\hspace{1mm}\begin{tikzpicture}[>=latex',baseline={(0,-.18)}]
  \draw[->] (0,.1) -- +(1,0);
  \node at (.5,0) {$\scriptscriptstyle\bot$};
  \draw[->] (1,-.1) -- +(-1,0);
  \draw[->] (1,-.2) -- +(-1,0);
\end{tikzpicture}\hspace{1mm}}

\newcommand{\tcof}{\mathsf{tcof}}
\newcommand{\tfib}{\mathsf{tfib}}
\newcommand{\cok}{\mathrm{cok}}

\newtheorem*{thm*}{\textbf{Theorem}}

\title{Characterizations of abstract stable homotopy theories}

\author{Moritz Groth}
\address{Rheinische Friedrich-Wilhelms-Universit{\"a}t Bonn, Mathematisches Institut, Ende-nicher Allee 60, 53115 Bonn, Germany}
\email{mgroth@math.uni-bonn.de}

\date{\today}

\begin{document}

\begin{abstract}
In this paper we establish new characterizations of stable derivators, thereby obtaining additional interpretations of the passage from (pointed) topological spaces to spectra and, more generally, of the stabilization. We show that a derivator is stable if and only if homotopy finite limits and homotopy finite colimits commute, and there are variants for sufficiently finite Kan extensions. As an additional reformulation, a derivator is stable if and only if it admits a zero object and if partial cone and partial fiber morphisms commute on squares.
\end{abstract}

\maketitle

\tableofcontents

\section{Introduction}
\label{sec:intro}

The classical paradigm of algebraic topology consists of studying topological spaces through algebraic invariants. Such invariants include ordinary cohomology theories, various flavors of topological $K$-theory and cobordism and other generalized cohomology theories such as stable homotopy groups (despite the latter often being too hard to calculate). By the classical Brown representability theorem such generalized cohomology theories are represented by \emph{spectra}, and a good part of modern algebraic topology consists of studying spectra or even the totality of all spectra.

A close connection between topological spaces and spectra is provided by the construction of suspension spectra, $X\mapsto\Sigma^\infty_+X$. In more detail, this construction factors into two intermediate steps,
\[
(\Sigma^\infty_+,\Omega^\infty_-)\colon\mathrm{Top}\rightleftarrows\mathrm{Top}_\ast\rightleftarrows\mathrm{Sp},
\]
where we denote by $\mathrm{Top},\mathrm{Top}_\ast,$ and $\mathrm{Sp}$ the homotopy theories of topological spaces, pointed topological spaces, and spectra, respectively. The first step simply adds a disjoint basepoint while the second step forms the suspension spectrum of a pointed topological space. (For a refined picture of this stabilization see \cite{ggn:infinite} and this will be revisited in \cite{gs:enriched}.) 

From a more abstract perspective, each of these two steps improves universally certain \emph{exactness properties} of the homotopy theory of (pointed) topological spaces. In the first step we pass in a universal way from a general homotopy theory to a \emph{pointed} homotopy theory, i.e., a homotopy theory admitting a zero object. The second step realizes the universal passage from a pointed homotopy theory to a \emph{stable} homotopy theory, i.e., to a pointed homotopy theory in which homotopy pushouts and homotopy pullbacks coincide. With these results in mind, our main goal in this paper is to collect additional answers to the following question.

\vspace{2mm}
\textbf{Question:} Which exactness properties of the homotopy theory of spectra already \emph{characterize} the passage from (pointed) topological spaces to spectra? To put it differently, starting with the homotopy theory of (pointed) topological spaces, for which exactness properties is it true that if one imposes these properties in a universal way then the outcome is the homotopy theory of spectra?
\vspace{2mm}

In order to make this question more specific, we stick to a precise definition of an abstract homotopy theory and here we choose to work with derivators. (However, using similar arguments, the results in this paper can also be obtained if one works with $\infty$-categories instead.) For the remainder of this introduction it suffices to know that derivators provide some framework for the calculus of homotopy limits, homotopy colimits, and homotopy Kan extensions as it is available in typical situations arising in homological algebra and abstract homotopy theory (but see, for instance, \cite{groth:intro-to-der-1} for more details).

A derivator is by definition \emph{stable} if it admits a zero object and if the classes of pullback squares and pushout squares coincide. Typical examples are given by derivators of unbounded chain complexes in Grothendieck abelian categories (like derivators associated to fields, rings, or schemes), and homotopy derivators of stable model categories or stable $\infty$-categories (see \cite[\S5]{gst:basic} for many explicit examples). The universal example of a stable derivator is given by the derivator of spectra, and this derivator is obtained by stabilizing the derivator of spaces \cite{heller:stable}. 

It is known that stability can be reformulated by asking that the derivator is pointed, i.e., that there is a zero object, and that the suspension-loop adjunction or the cofiber-fiber adjunction is an equivalence \cite{gps:mayer}. An additional characterization of stability was established in \cite{gst:basic}, namely as pointed derivators in which the classes of strongly cartesian $n$-cubes (in the sense of Goodwillie \cite{goodwillie:II}) and strongly cocartesian $n$-cubes agree for all $n\geq 2$. 

In this paper we characterize stable derivators as precisely those derivators in which homotopy finite limits and homotopy finite colimits commute. (Let us recall that a category is homotopy finite if it is equivalent to a category which is finite, skeletal, and has no non-trivial endomorphisms, i.e., to a category whose nerve is a finite simplicial set.) Since Kan extensions in derivators are calculated pointwise, these characterizations admit various improvements in terms of the commutativity of sufficiently finite Kan extensions (in the sense of \cite[\S9]{groth:can-can}). 

An additional reformulation can be obtained in terms of partial cone and partial fiber morphisms acting on squares in pointed derivators. In more detail, since a square can be read as a morphism of morphisms, we can form the cone in one direction and the fiber in the other direction, and a pointed derivator is stable if and only if these two operations commute. (In contrast, the formations of cones in the two different directions commute in every pointed derivator and these compositions are shown to agree with the total cofiber construction; see \S\ref{sec:tcof}.)  As a summary, the following are the main characterizations established as \autoref{thm:stable-lim-III}.

\begin{thm*}
The following are equivalent for a derivator \D.
\begin{enumerate}
\item The derivator \D is stable.
\item The derivator \D is pointed and the cone morphism $C\colon\D^{[1]}\to\D$ preserves fibers. (Here, $\D^{[1]}$ denotes the derivator of morphisms in \D.)
\item Homotopy finite colimits and homotopy finite limits commute in \D.
\item Left homotopy finite left Kan extensions commute with arbitrary right Kan extensions in \D.
\item Arbitrary left Kan extensions commute with right homotopy finite right Kan extensions in \D. 
\end{enumerate}
\end{thm*}

Since the derivator of spectra is the stabilization of the derivator of spaces, these abstract characterizations of stability  specialize to answers to the above question. For concreteness, there are the following answers while additional, closely related ones can be found in \S\ref{sec:char}.

\vspace{2mm}
\textbf{Answer:} The derivator of spectra is obtained from the derivator of spaces if one forces homotopy finite limits and homotopy finite colimits to commute in a universal way. Similarly, the derivator of spectra is obtained from the derivator of pointed spaces if one forces partial cones and partial fibers on squares to commute in a universal way.
\vspace{2mm}

This paper belongs to a project aiming for an abstract study of stability, and the paper can be thought of as a sequel to \cite{groth:ptstab,gps:mayer,groth:can-can} and as a prequel to \cite{groth:formal}. This abstract study of stability was developed in a different direction in the series of papers on abstract representation theory \cite{gst:basic,gst:tree,gst:Dynkin-A,gst:acyclic} which will be continued in \cite{gst:acyclic-Serre}. The perspective from enriched derivator theory offers additional characterizations of stability, and these together with a more systematic study of the stabilization will appear in \cite{gs:enriched}.

The content of the sections is as follows. In \S\ref{sec:tcof} we study partial cones, iterated cones, and total cofibers in pointed derivators and show the latter two to be canonically isomorphic. In \S\ref{sec:char} we characterize pointed and stable derivators by the commutativity of certain (co)limits or Kan extensions. In \S\ref{sec:fun} we obtain additional characterizations in terms of iterated adjoints to constant morphism morphisms.
\vspace{2mm}

\textbf{Prerequisites.} We assume \emph{basic} acquaintance with the language of derivators, which were introduced independently by Grothendieck~\cite{grothendieck:derivators}, Heller~\cite{heller:htpythies}, and Franke~\cite{franke:adams}. Derivators were developed further by various mathematicians including Maltsiniotis \cite{maltsiniotis:seminar,maltsiniotis:k-theory,maltsiniotis:htpy-exact} and Cisinski \cite{cisinski:direct,cisinski:loc-min,cisinski:derived-kan} (see \cite{grothendieck:derivators} for many additional references). Here we stick to the notation and conventions from \cite{gps:mayer}. For a more detailed account of the basics we refer to \cite{groth:intro-to-der-1}.

\section{Partial cones, iterated cones, and total cofibers}
\label{sec:tcof}

In this section we define total cofibers and iterated cones of squares in pointed derivators, and show these constructions to be naturally isomorphic. These and related results belong to a fairly rich calculus of squares, cubes, and hypercubes in pointed derivators, which we intend to come back to elsewhere. 

We refer the reader to \cite[\S2]{groth:can-can} for the construction of canonical comparison maps between cocones and colimiting cocones. In the special case in which we start with $\ulcorner=\square-\{(1,1)\}$ the full subcategory of the square obtained by removing the final object, one is lead to consider the category $P=\square^\rhd$. This is the cocone on the square obtained by adjoining a new terminal object $\infty$,
\[
\xymatrix{
(0,0)\ar[r]\ar[d]&(1,0)\ar[d]&\\
(0,1)\ar[r]&(1,1)\ar[dr]&\\
&&\infty.
}
\]
Associated to this category there are the fully faithful inclusions of the source and target square 
\begin{equation}\label{eq:t-s-squares}
s=s_\ulcorner\colon\square\to P=\square^\rhd\qquad\text{and}\qquad t=t_\ulcorner\colon\square\to P=\square^\rhd,
\end{equation}
and this datum corepresents morphisms of cocones on spans.

\begin{prop}\label{prop:cocart-cocone}
Let \D be a derivator and let $s,t\colon\square\to P=\square^\rhd$ be the inclusions of the source and target squares.
\begin{enumerate}
\item The morphism $t_!\colon\D^\square\to\D^P$ is fully faithful and $Y\in\D^P$ lies in the essential image of~$t_!$ if and only if the source square $s^\ast Y$ is cocartesian.
\item A square $X\in\D^\square$ is cocartesian if and only if the following canonical comparison map is an isomorphism,
\begin{equation}\label{eq:comp-map}
\mathrm{can}=\mathrm{can}(X)\colon t_!(X)_{1,1}\to t_!(X)_\infty.
\end{equation}
\end{enumerate}
\end{prop}
\begin{proof}
This is a special case of \cite[Prop.~3.11]{groth:can-can} and \cite[Prop.~3.14]{groth:can-can}.
\end{proof}

\begin{defn}\label{defn:tcof}
Let \D be a pointed derivator. The \textbf{total cofiber} of $X\in\D^\square$ is the cone of the comparison map~\eqref{eq:comp-map}. In formulas we set
\[
\tcof(X)=C(\mathrm{can}(X))\in\D.
\]
\end{defn}

The definition of the \textbf{total fiber} $\tfib(X)\in\D$ is dual. It is immediate from the above construction that there are morphisms of derivators
\[
\tcof\colon\D^\square\to\D\qquad\text{and}\qquad\tfib\colon\D^\square\to\D.
\]

\begin{lem}\label{lem:tcof-adjoint}
In every pointed derivator \D the morphism $\tcof\colon\D^\square\to\D$ is a left adjoint and the functor $\tfib\colon\D^\square\to\D$ is a right adjoint.
\end{lem}
\begin{proof}
Let $j\colon[1]=(0<1)\to P=\square^\rhd$ be the functor classifying the morphism $(1,1)\to \infty$. It follows from the definition of the total cofiber (\autoref{defn:tcof}) that $\mathsf{tcof}$ is the composition
\begin{equation}\label{eq:tcof-defn}
\D^\square\stackrel{t_!}{\to}\D^P\stackrel{j^\ast}{\to}\D^{[1]}\stackrel{\cof}{\to}\D^{[1]}\stackrel{1^\ast}{\to}\D.
\end{equation}
The morphisms $t_!,j^\ast,$ and $1^\ast$ are obviously left adjoint functors as is $\cof$  (\cite[Prop.~3.20]{groth:ptstab}).
\end{proof}

\begin{rmk}
Composing the respective right adjoints to each of the morphisms in \eqref{eq:tcof-defn}, one checks that a right adjoint to $\tcof$ is given by $(1,1)_!\colon\D\to\D^\square$, i.e., by the left extension by zero morphism \cite[Prop.~3.6]{groth:ptstab} which sends $x\in\D$ to a coherent square looking like 
\[
\xymatrix{
0\ar[r]\ar[d]&0\ar[d]\\
0\ar[r]&x.
}
\]
The adjunction
\[
(\tcof,(1,1)_!)\colon\D^\square\rightleftarrows\D
\]
exhibits $\tcof$ as an \emph{exceptional inverse image morphism} \cite[\S3.1]{groth:ptstab}. Alternatively, this also follows from the explicit formulas for exceptional inverse images from \emph{loc.~cit.} together with a formula for colimits of punctured cubes.
\end{rmk}

The total cofiber is an obstruction against a square being cocartesian.

\begin{cor}\label{cor:cocart-tcof-zero}
Let \D be a pointed derivator and let $X\in\D^\square$. If $X$ is cocartesian, then $\tcof(X)\cong 0$.
\end{cor}
\begin{proof}
By \autoref{prop:cocart-cocone} a square $X\in\D^\square$ is cocartesian if and only if the canonical morphism $\mathrm{can}(X)\in\D^{[1]}$ is an isomorphism. Since cones of isomorphisms are trivial \cite[Prop.~3.12]{groth:ptstab}, the claim follows from \autoref{defn:tcof}.
\end{proof}

We now turn to iterated cone constructions and show them to be naturally isomorphic to total cofibers. Let \D be a pointed derivator and let $X\in\D^\square$ be a coherent square looking like
\begin{equation}\label{eq:square-appl}
\vcenter{
\xymatrix{
x\ar[r]^-f\ar[d]_-{g}&y\ar[d]^-{g'}\\
x'\ar[r]_-{f'}&y'.
}
}
\end{equation}
Reading such a square as a morphism of morphisms in two different ways, we can form the following two `partial cofiber cubes'.

\begin{con}\label{con:cone-1-2}
Let \D be a pointed derivator and let us consider the morphism which sends a coherent morphism to its cofiber square. In more detail, denoting by $i\colon[1]\to\ulcorner$ the functor classifying the horizontal morphism $(0,0)\to(1,0)$, we consider the morphism of derivators
\[
\D^{[1]}\stackrel{i_\ast}{\to}\D^\ulcorner\stackrel{(i_\ulcorner)_!}{\to}\D^\square.
\]
Forming such cofiber squares in the first or in the second coordinate, we obtain morphisms
\[
c_1\colon\D^{[1]\times[1]}\to\D^{\square\times[1]}\qquad\text{and}\qquad c_2\colon\D^{[1]\times[1]}\to\D^{[1]\times\square},
\]
which send a coherent square $X\in\D^\square$ looking like \eqref{eq:square-appl} to coherent cubes $c_1(X)$ and $c_2(X)$ with respective underlying diagrams
\begin{equation}\label{eq:cone-1-2-cubes}
\vcenter{
\xymatrix@-1pc{
x \ar[rr]^-f \ar[dr]_-g \ar[dd] && y \ar[dr]^-{g'} \ar'[d][dd] \\
& x' \ar[rr]_-(.2){f'} \ar[dd] && y' \ar[dd]\\
0 \ar[dr] \ar'[r][rr] && Cf \ar[dr]\\
& 0 \ar[rr] && Cf',
}
}
\qquad\qquad
\vcenter{
\xymatrix@-1pc{
x \ar[rr]^-f \ar[dr]_-g \ar[dd] && y \ar[dr]^-{g'} \ar'[d][dd] \\
& x' \ar[rr]_-(.2){f'} \ar[dd] && y' \ar[dd]\\
0 \ar[dr] \ar'[r][rr] && 0 \ar[dr]\\
& Cg \ar[rr] && Cg'.
}
}
\end{equation}
In the cube $c_1(X)$ on the left the back and the front squares are cocartesian while in the cube $c_2(X)$ on the right this is the case for the left and the right squares.
In particular, associated to $X\in\D^\square$ as in~\eqref{eq:square-appl} there are coherent maps $C_1(X),C_2(X)\in\D^{[1]},$
\begin{equation}\label{eq:cone-1-2}
C_1(X)\colon Cf\to Cf'\qquad\text{and}\qquad C_2(X)\colon Cg\to Cg'.
\end{equation}
We refer to these morphisms as \textbf{partial cones} of $X$, and there are corresponding partial cone morphisms
\[
C_1,C_2\colon\D^\square\to\D^{[1]}.
\]
Composing these partial cones with the usual cone morphism $C\colon\D^{[1]}\to\D$, we obtain \textbf{iterated cone morphisms}
\[
C\circ C_1, C\circ C_2\colon\D^\square\to\D.
\]
\end{con}

\begin{prop}\label{prop:cocart-map-on-C}
Let \D be a pointed derivator and let $X\in\D^\square$. If $X$ is cocartesian, then the partial cones $C_1(X),C_2(X)\in\D^{[1]}$ are isomorphisms.
\end{prop}
\begin{proof}
In order to show that the partial cone $C_1(X)$ is an isomorphism, let us consider the defining cube on the left in \eqref{eq:cone-1-2-cubes}. By construction the back and front faces are cocartesian as is the top face by assumption on~$X$. The composition and cancellation property of cocartesian squares \cite[Prop.~3.13]{groth:ptstab} implies that also the bottom face is cocartesian. Since isomorphisms are stable under cobase change \cite[Prop.~3.12]{groth:ptstab}, the partial cone $C_1(X)\colon Cf\to Cf'$ is an isomorphism. Similar arguments show that also the partial cone $C_2(X)\colon Cg\to Cg'$ is an isomorphism.
\end{proof}

\begin{rmk}
Let $y_\cC$ be the pointed derivator represented by a complete, cocomplete, and pointed category~$\cC$. In this case \autoref{prop:cocart-map-on-C} reduces to the statement that for every pushout square in $\cC$,
\[
\xymatrix{
x\ar[r]^-f\ar[d]_-{g}&y\ar[d]^-{g'}\\
x'\ar[r]_-{f'}&y',\pushoutcorner
}
\]
the induced maps $\cok(f)\to\cok(f')$ and $\cok(g)\to\cok(g')$ are isomorphisms.
\end{rmk}

\begin{cor}\label{cor:cocart-iterated-cones}
Let \D be a pointed derivator and let $X\in\D^\square$. If $X$ is cocartesian, then there are isomorphisms $C(C_1(X))\cong 0\cong C(C_2(X))$.
\end{cor}
\begin{proof}
By \autoref{prop:cocart-map-on-C} we know that $C_1(X),C_2(X)$ are isomorphisms as soon as $X$ is cocartesian. Hence it suffices to remember that cones of isomorphisms are trivial \cite[Prop.~3.12]{groth:ptstab}.
\end{proof}

In particular, for cocartesian squares we just observed that iterated cones are isomorphic. This is true more generally.

\begin{prop}\label{prop:cones-commute}
For every pointed derivator \D the iterated cones
\[
C\circ C_1\colon\D^\square\to\D\qquad\text{and}\qquad C\circ C_2\colon\D^\square\to\D 
\]
are canonically isomorphic.
\end{prop}
\begin{proof}
As a left adjoint morphism, the cone morphism $C\colon\D^{[1]}\to\D$ is right exact and hence preserves cones \cite[\S9]{groth:can-can}. This precisely means that iterated cones are canonically isomorphic.
\end{proof}

\begin{rmk}
\begin{enumerate}
\item The calculus of hypercubes in pointed derivators yields an additional proof leading to a more precise statement. In particular, this also implies that given a square $X\in\D^\square$ in pointed derivator \D which looks like \eqref{eq:square-appl} there is a square of the form
\[
\xymatrix{
y'\ar[r]\ar[d]&Cf'\ar[d]\\
Cg'\ar[r]&\tcof(X).
}
\]
\item Passing to shifted derivators of pointed derivators, it is immediate that partial cone operations defined on hypercubes also commute. 
\item As noted in the proof, the cone morphism $C\colon\D^{[1]}\to\D$ preserve cones, and there is a dual result for fibers. However, a pointed derivator is stable if and only if the cone morphism preserves fibers; see \S\ref{sec:char}.
\end{enumerate}
\end{rmk}

Iterated cones and total cofibers vanish on cocartesian squares (\autoref{cor:cocart-tcof-zero} and \autoref{cor:cocart-iterated-cones}), and are hence naturally isomorphic on such squares. Again, this is true more generally. 

\begin{thm}\label{thm:total-cof}
Total cofibers and iterated cones in pointed derivators \D are naturally isomorphic,
\[
\tcof\cong C\circ C_1\cong C\circ C_2\colon\D^\square\to\D.
\]
\end{thm}
\begin{proof}
Let $X\in\D^\square$ be a square looking like~\eqref{eq:square-appl}. We show that suitable combinations of Kan extensions can be used to construct a coherent diagram as in \autoref{fig:tcof-CC}. To this end, we denote by $[n],n\geq 0,$ the poset $(0<1<\ldots<n)$ considered as a category. Let $B\subseteq[1]\times[2]\times[2]$ be the full subcategory given by that figure, in which the third coordinate is drawn vertically. There is a fully faithful functor
$i\colon\square\to B$ given by 
\[
(0,0)\mapsto(0,0,0),\quad (1,0)\mapsto(1,0,0),\quad (0,1)\mapsto(0,1,0),\quad (1,1)\mapsto(1,2,0).
\]
As we describe next, this functor factors as a composition of fully faithful functors
\[
i\colon \square\stackrel{i_1}{\to}B_1\stackrel{i_2}{\to}B_2\stackrel{i_3}{\to}B_3\stackrel{i_4}{\to}B_4\stackrel{i_5}{\to}B_5\stackrel{i_6}{\to}B,
\]
where all intermediate categories are obtained by adding certain objects to the image of~$i$. The corresponding Kan extension morphisms
\[
\D^\square\stackrel{(i_1)_!}{\to}\D^{B_1}\stackrel{(i_2)_!}{\to}\D^{B_2}\stackrel{(i_3)_\ast}{\to}\D^{B_3}\stackrel{(i_4)_!}{\to}\D^{B_4}\stackrel{(i_5)_\ast}{\to}\D^{B_5}\stackrel{(i_6)_!}{\to}\D^B
\]
then build the desired diagram. Since these Kan extension morphisms are fully faithful \cite[Prop.~1.20]{groth:ptstab}, it suffices to understand the individual steps.
\begin{figure}
\begin{equation}
\vcenter{
\xymatrix@-1.5pc{
&&x\ar[rrr]^-f\ar[dl]_-g\ar[ddd]&&&y\ar[dl]^-{\tilde g}\ar[ddd]\\
&x'\ar[rrr]_-{\tilde f}\ar[dl]_-=\ar[ddd]&&&p\pushoutcorner\ar[dl]^-{\mathrm{can}}\ar[ddd]&\\
x'\ar[rrr]_-{f'}\ar[ddd]&&&y'\ar[ddd]&&\\
&&0_1\ar[rrr]\ar[dl]&&&Cf\ar[dl]^-\cong\ar[ddd]\\
&0_2\ar[rrr]\ar[dl]&&&C\tilde f\ar[dl]\ar[ddd]&\\
0_3\ar[rrr]&&&Cf'\ar[ddd]&&\\
&&&&&0_4\ar[dl]\\
&&&&0_5\ar[dl]&\\
&&&c(X).&&
}
}
\end{equation}
\caption{Total cofiber versus iterated cones}
\label{fig:tcof-CC}
\end{figure}
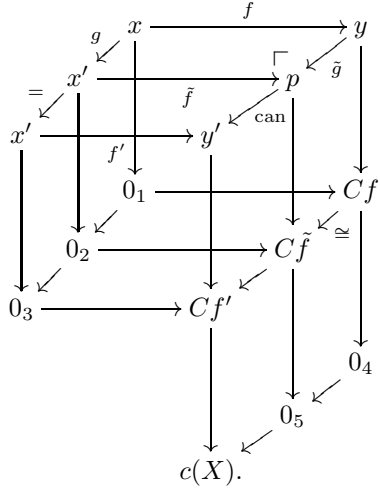
\begin{enumerate}
\item The category $B_1$ is obtained by adding the object $(1,1,0)$. Obviously, the functor $i_1\colon\square\to B_1$ is isomorphic to the functor $t$ in~\eqref{eq:t-s-squares}, and \autoref{prop:cocart-cocone} shows that $(i_1)_!$ adds a cocartesian square and the comparison map. 
\item The category $B_2$ also contains the object $(0,2,0)$ and is hence isomorphic to $[1]\times[2]$. To calculate the Kan extension morphism $(i_2)_!\colon\D^{B_1}\to\D^{B_2}$, we invoke the pointwise formula for Kan extensions and observe that the slice category under consideration is isomorphic to $[1]$. Since this category admits $1\in[1]$ as terminal object, we conclude by \cite[Lem.~1.19]{groth:ptstab} that $(i_2)_!$ essentially adds the identity morphism $\id_{x'}$ and forms the composition $f'=\mathrm{can}\circ \tilde f$.
\item The category $B_3$ is obtained by adding the objects $(0,0,1),(0,1,1),(0,2,1)$. Note that the inclusion $i_3\colon B_2\to B_3$ is isomorphic to 
\[
k\times\id\colon[1]\times[2]\to\ulcorner\times[2],
\]
where $k$ classifies the horizontal morphism $(0,0)\to(1,0)$. Since $k$ is a sieve, so is the functor $i_3$ and $(i_3)_\ast$ is hence right extension by zero \cite[Prop.~3.6]{groth:ptstab}. Thus, this right Kan extension morphism adds the zero objects $0_1,0_2,0_3$.
\item The category $B_4$ also contains the objects $(1,0,1),(1,1,1),(1,2,1)$, hence the inclusion $i_4$ is isomorphic to the functor
\[
i_\ulcorner\times\id\colon\ulcorner\times[2]\to\square\times[2].
\]
By the basic calculus of parametrized Kan extensions in derivators (\cite[Cor.~2.6]{groth:ptstab}), the morphism $(i_4)_!\colon\D^{B_3}\to\D^{B_4}$ hence forms three cocartesian squares. Applied to the essential image of the previous Kan extension functors, it thus forms three cofiber squares, thereby adding the objects $Cf,C\tilde f,$ and $Cf'.$ Since the square populated by $x\text{-}y\text{-}x'\text{-}p$ is cocartesian, the map $Cf\toiso C\tilde f$ is by \autoref{prop:cocart-map-on-C} an isomorphism.
\item The category $B_5$ is obtained by adding the objects $(1,0,2)$ and $(1,1,2)$. Since the inclusion $i_5\colon B_4\to B_5$ is a sieve, $(i_5)_\ast$ is right extension by zero \cite[Prop.~3.6]{groth:ptstab}) and it hence adds the zero objects $0_4,0_5$.
\item The final step adds the remaining object $(1,2,2)$, and it suffices to understand $(i_6)_!\colon\D^{B_5}\to\D^B$. A homotopy finality argument (based on the detection result \cite[Prop.~3.10]{groth:ptstab}) shows that $(i_6)_!$ precisely amounts to forming a cocartesian square $C\tilde f\text{-}Cf'\text{-}0_5\text{-}c(X)$. 
\end{enumerate}
This concludes the functorial construction of a diagram $Q=Q(X)\in\D^B$ as in \autoref{fig:tcof-CC}. The point of this diagram is that it allows us to identify $c(X)\in\D$ in the following two ways.
\begin{enumerate}
\item As already observed, in this diagram the square $C\tilde f\text{-}Cf'\text{-}0_5\text{-}c(X)$ is cocartesian, and we next show that the same is true for $p\text{-}y'\text{-}C\tilde f\text{-}Cf'$. To this end, we note that $x'\text{-}x'\text{-}0_2\text{-}0_3$ is cocartesian as a horizontally constant square \cite[Prop.~3.12]{groth:ptstab} and that the horizontally displayed squares are cofiber squares. Hence, the cancellation property of cocartesian squares \cite[Cor.~4.10]{gps:mayer} shows that $p\text{-}y'\text{-}C\tilde f\text{-}Cf'$ is cocartesian. Finally, the composition property of these squares (see \emph{loc.~cit.}) implies that also the square $p\text{-}y'\text{-} 0_5\text{-}c(X)$ is cocartesian, and this cofiber square shows that there is a canonical isomorphism $c(X)\cong C(p\to y')\cong\mathsf{tcof}(X)$.
\item Moreover, the square $Cf\text{-}C\tilde{f}\text{-}0_4\text{-}0_5$ is horizontally constant and hence cocartesian. Since $C\tilde f\text{-}Cf'\text{-}0_5\text{-}c(X)$ is also cocartesian, so is $Cf\text{-}Cf'\text{-}0_4\text{-}c(X)$, yielding canonical isomorphisms $c(X)\cong C(Cf\to Cf')\cong C(C_1(X))$.
\end{enumerate}
Taking these two steps together, we obtain a natural isomorphism $\tcof\cong C\circ C_1$, and by symmetry we obtain a similar natural isomorphism $\tcof\cong C\circ C_2$.
\end{proof}

\begin{rmk}
An alternative proof of \autoref{thm:total-cof} is obtained by identifying cones, partial cones, and total cofibers as \emph{exceptional inverse image morphisms} (\cite[\S3.1]{groth:ptstab}). As left adjoints to left Kan extension morphisms, these morphisms are compatible with compositions up to canonical isomorphism.
\end{rmk}

\begin{cor}
The following are equivalent for a square $X\in\D^\square$ in a stable derivator.
\begin{enumerate}
\item The square $X$ is bicartesian.
\item The partial cone $C_1X$ is an isomorphism.
\item The partial cone $C_2X$ is an isomorphism.
\item The total cofiber $\tcof(X)$ is trivial.
\item The partial fiber $F_1X$ is an isomorphism.
\item The partial fiber $F_2X$ is an isomorphism.
\item The total fiber $\tfib(X)$ is trivial.
\end{enumerate}
\end{cor}
\begin{proof}
By duality it suffices to show that the first four statements are equivalent. Since in stable derivators isomorphisms are detected by the triviality of the cone \cite[Prop.~4.5]{groth:ptstab}, the statements (ii),(iii), and (iv) are equivalent by \autoref{thm:total-cof}. Similarly, also the equivalence of (i) and (iv) is immediate from this and \autoref{prop:cocart-cocone}.
\end{proof}

\section{Stability and commuting (co)limits}
\label{sec:char}

In this section we obtain characterizations of pointed and stable derivators in terms of the commutativity of certain left and right Kan extensions. It turns out that a derivator is stable if and only if homotopy finite colimits and homotopy finite limits commute, and there are variants using suitable Kan extensions.  

We begin by collecting the following characterizations which already appeared in the literature.

\begin{thm}\label{thm:stable-known}
The following are equivalent for a pointed derivator \D.
\begin{enumerate}
\item The adjunction $(\Sigma,\Omega)\colon\D\rightleftarrows\D$ is an equivalence.
\item The derivator \D is $\Sigma$-stable, i.e., a square in \D is a suspension square if and only if it is a loop square.
\item The adjunction $(\cof,\fib)\colon\D^{[1]}\rightleftarrows\D^{[1]}$ is an equivalence.
\item The derivator \D is cofiber-stable, i.e., a square in \D is a cofiber square if and only if it is a fiber square.
\item The derivator \D is stable, i.e., a square in \D is cocartesian if and only if it is cartesian.
\item An $n$-cube in \D, $n\geq 2,$ is strongly cocartesian if and only if it is strongly cartesian.
\end{enumerate}
\end{thm}
\begin{proof}
The equivalence of the first five statements is \cite[Thm.~7.1]{gps:mayer} and the equivalence of the remaining two is \cite[Cor.~8.13]{gst:tree}.
\end{proof}

As a preparation for a minor variant we include the following construction.

\begin{con}
In every pointed derivator \D there are canonical comparison maps
\begin{equation}\label{eq:sigma-f-c}
\Sigma F\to C\colon\D^{[1]}\to\D\qquad\text{\and}\qquad F\to \Omega C\colon\D^{[1]}\to\D.
\end{equation}
In fact, starting with a morphism $(f\colon x\to y)\in\D^{[1]}$ we can pass to the coherent diagram encoding both the corresponding fiber and cofiber square,
\[
\xymatrix{
Ff\ar[r]\ar[d]\pullbackcorner&x\ar[d]^-f\ar[r]&0\ar[d]\\
0\ar[r]&y\ar[r]&Cf.\pushoutcorner
}
\]
More formally, let $i\colon[1]\to\boxbar=[2]\times[1]$ classify the vertical morphism in the middle and let 
\[
i\colon [1]\stackrel{i_1}{\to}A_1\stackrel{i_2}{\to}A_2\stackrel{i_3}{\to}A_3\stackrel{i_4}{\to}\boxbar
\]
be the fully faithful inclusions which in turn add the objects $(2,0),(2,1),(0,1),$ and $(0,0)$. In every pointed derivator we can consider the corresponding Kan extension morphisms
\[
\D^{[1]}\stackrel{(i_1)_\ast}{\to}\D^{A_1}\stackrel{(i_2)_!}{\to}\D^{A_2}\stackrel{(i_3)_!}{\to}\D^{A_3}\stackrel{(i_4)_\ast}{\to}\D^\boxbar.
\]
The first two functors add a cofiber square and a homotopy cofinality argument (for example based on \cite[Prop.~3.10]{groth:ptstab}) shows that the remaining two morphisms add the fiber square. Forming the composite square, we obtain a coherent square looking like
\begin{equation}\label{eq:F-C-square}
\vcenter{
\xymatrix{
Ff\ar[r]\ar[d]&0\ar[d]\\
0\ar[r]&Cf.
}
}
\end{equation}
The canonical comparison maps \eqref{eq:sigma-f-c} are respectively special cases of the comparison maps from \autoref{prop:cocart-cocone} or its dual applied to the above square.
\end{con}

\begin{prop}\label{prop:stable-known-mod}
The following are equivalent for a pointed derivator \D.
\begin{enumerate}
\item The derivator \D is stable.
\item For every $f\in\D^{[1]}$ the canonical comparison maps $\Sigma F\to C$ and $F\to \Omega C$ as in \eqref{eq:sigma-f-c} are isomorphisms.
\item For every $f\in\D^{[1]}$ the square \eqref{eq:F-C-square} is bicartesian.
\end{enumerate}
\end{prop}
\begin{proof}
If \D is a stable derivator, then the composition property of bicartesian squares \cite[Prop.~3.13]{groth:ptstab} implies that \eqref{eq:F-C-square} is bicartesian, and it follows from \autoref{prop:cocart-cocone} that the canonical transformations $\Sigma F\to C$ and $F\to\Omega C$ are invertible. It remains to show that (ii) implies (iii), and we hence assume that $\Sigma F\toiso C$ is invertible. Associated to $x\in\D$ there is by \cite[Prop.~3.6]{groth:ptstab} the morphism $1_!(x)=(0\to x)\in\D^{[1]}$. The natural isomorphism $\Sigma F\toiso C$ evaluated at $1_!(x)$ yields a natural isomorphism $\Sigma\Omega x\toiso x$. Dually, we deduce $\id\toiso\Omega\Sigma$ and \autoref{thm:stable-known} concludes the proof.
\end{proof}

While unrelated left Kan extensions always commute \cite[Cor.~4.3]{groth:can-can}, it is, in general, not true that unrelated left and right Kan extensions commute. More specifically, given functors $u\colon A\to A'$ and $v\colon B\to B'$, recall that \textbf{left Kan extension along $u$ and right Kan extension along $v$ commute} in a derivator \D if the canonical mate
\begin{align}\label{eq:lkan-rkan-comm}
(u\times\id)_!(\id\times v)_\ast&\stackrel{\eta}{\to} (u\times\id)_!(\id\times v)_\ast (u\times\id)^\ast(u\times\id)_!\\
&\toiso (u\times\id)_!(u\times\id)^\ast(\id\times v)_\ast (u\times\id)_!\\
& \stackrel{\varepsilon}{\to} (\id\times v)_\ast (u\times\id)_!
\end{align}
is an isomorphism in \D. This is to say that the morphism $u_!\colon\D^A\to\D^{A'}$ preserves right Kan extensions along $v$ or that the morphism $v_\ast\colon\D^B\to\D^{B'}$ preserves left Kan extensions along $u$ \cite[Lem.~4.8]{groth:can-can}. For the purpose of a simpler terminology, we also say that $u_!$ and $v_\ast$ commute in \D. 

In general, the canonical mates \eqref{eq:lkan-rkan-comm} are not invertible as is for example illustrated by the following characterization of pointed derivators.

\begin{prop}\label{prop:ptd-comm}
The following are equivalent for a derivator \D.
\begin{enumerate}
\item The derivator \D is pointed.
\item Empty colimits and empty limits commute in \D.
\item Left Kan extensions along cosieves and right Kan extensions along sieves commute in \D.
\item[(iv.a)] Left Kan extensions along cosieves and arbitrary right Kan extensions commute in \D. 
\item[(iv.b)] Arbitrary left Kan extensions and right Kan extensions along sieves commute in \D.
\end{enumerate}
\end{prop}
\begin{proof}
For the equivalence of the first two statements we consider the empty functor $\emptyset\colon\emptyset\to\bbone$. Correspondingly, for every derivator \D there is the canonical mate
\[
\xymatrix{
\D^{\emptyset\times\emptyset}\ar[r]^-{(\id\times\emptyset)_\ast}\ar[d]_-{(\emptyset\times\id)_!}\drtwocell\omit{}&
\D^{\emptyset\times\bbone}\ar[d]^--{(\emptyset\times\id)_!}\\
\D^{\bbone\times\emptyset}\ar[r]_--{(\id\times\emptyset)_\ast}&\D^{\bbone\times\bbone}
}
\]
detecting if empty colimits and empty limits commute. By construction of initial and final objects in derivators (see \cite[\S1.1]{groth:ptstab}), the source of this canonical mate is given by initial objects in \D while the target is given by final objects. Hence, \D is pointed if and only if empty colimits and empty limits commute in \D.

Obviously, each of the statements (iv.a) or (iv.b) implies statement (iii). Moreover, since the empty functor is a sieve and a cosieve, statement (iii) implies (ii). By duality, it only remains to show that (i) implies (iv.a). Given a functor $u\colon A\to B$, the morphism $u_\ast\colon\D^A\to\D^B$ is a right adjoint and, as a pointed morphism of pointed derivators, $u_\ast$ hence preserves left Kan extensions along cosieves \cite[Cor.~8.2]{groth:can-can}.
\end{proof}

We now turn to the stable context. Let us recall that a category $A\in\cCat$ is \textbf{strictly homotopy finite} if it is finite, skeletal, and it has no non-trivial endomorphisms (equivalently the nerve $NA$ is a finite simplicial set). A category is \textbf{homotopy finite} if it is equivalent to a strictly homotopy finite category.

\begin{thm}\label{thm:stable-lim-I}
Homotopy finite colimits and homotopy finite limits commute in stable derivators.
\end{thm}
\begin{proof}
Let \D be a stable derivator and let $A$ be a homotopy finite category. Denoting by $\pi_A\colon A\to\bbone$ the unique functor, there are defining adjunctions 
\[
(\colim_A,\pi_A^\ast)\colon\D^A\rightleftarrows\D\qquad\text{and}\qquad (\pi_A^\ast,\mathrm{lim}_A)\colon\D\rightleftarrows\D^A,
\]
and these exhibit $\colim_A,\mathrm{lim}_A\colon\D^A\to\D$ as exact morphisms of stable derivators \cite[Cor.~9.9]{groth:can-can}. Hence, by \cite[Thm.~7.1]{ps:linearity}, $\colim_A$ preserves homotopy finite limits and $\lim_A$ preserves homotopy finite colimits.
\end{proof}

For the converse to this theorem we collect the following lemma. 

\begin{lem}\label{lem:lim-comm}
Let \D be a derivator such that homotopy finite colimits and homotopy finite limits commute in \D.
\begin{enumerate}
\item The derivator \D is pointed.
\item The morphisms $\cof\colon\D^{[1]}\to\D^{[1]}$ and $C\colon\D^{[1]}\to\D$ preserve homotopy finite limits.
\item The morphism $\fib\colon\D^{[1]}\to\D^{[1]}$ and $F\colon\D^{[1]}\to\D$ preserve homotopy finite colimits.
\end{enumerate}
\end{lem}
\begin{proof}
By assumption on \D, empty colimits and empty limits commute and this implies that \D is pointed (\autoref{prop:ptd-comm}). Hence, by duality, it remains to take care of the second statement. Denoting by $i\colon[1]\to\ulcorner$ the sieve classifying the horizontal morphism $(0,0)\to (1,0)$ and by $k'\colon[1]\to\square$ the functor classifying the vertical morphism $(1,0)\to (1,1)$, the cofiber morphism is given by
\begin{equation}\label{eq:cof}
\cof\colon\D^{[1]}\stackrel{i_\ast}{\to}\D^\ulcorner\stackrel{(i_\ulcorner)_!}{\to}\D^\square\stackrel{(k')^\ast}{\to}\D^{[1]}.
\end{equation}
Since the morphisms $i_\ast$ and $(k')^\ast$ are right adjoints, they preserve arbitrary right Kan extensions, hence homotopy finite limits. By assumption on \D, \cite[Prop.~3.9]{groth:can-can}, and \cite[Lem.~4.9]{groth:can-can}, also the morphism $(i_\ulcorner)_!$ preserves homotopy finite limits, and hence so does $\cof$ by \cite[Prop.~5.2]{groth:can-can}. An additional composition with the continuous evaluation morphism $1^\ast\colon\D^{[1]}\to\D$ establishes the corresponding result for $C$. 
\end{proof}

We recall the notation used in \autoref{con:cone-1-2}.

\begin{notn}\label{notn:cof-fib}
Let \D be a pointed derivator and let $X\in\D^\square$. Then we can form the cone in two different directions, leading to the morphisms $C_1(X),C_2(X)\in\D^{[1]}$. Similarly, we obtain $\cof_1(X),\cof_2(X)\in\D^\square$, and we use the corresponding notation for the dual constructions. 
\end{notn}

Given a pointed derivator \D and $X\in\D^\square$, since $\cof$ and $C$ are pointed morphisms of pointed derivators, there are canonical comparison maps
\begin{equation}\label{eq:cof-fib-comm}
\cof_1(\fib_2 X)\to\fib_2(\cof_1X)\qquad\text{and}\qquad C(F_2 X)\to F(C_1X);
\end{equation}
see \cite[Construction~9.7]{groth:can-can}.

\begin{cor}\label{cor:lim-comm}
Let \D be a derivator in which homotopy finite colimits and homotopy finite limits commute. Then \D is pointed and the canonical transformations \eqref{eq:cof-fib-comm} are isomorphisms for every $X\in\D^\square$.
\end{cor}
\begin{proof}
This is immediate from \autoref{lem:lim-comm}.
\end{proof}

As we show next, this property already implies that the derivator is stable. Together with \autoref{thm:stable-lim-I} we thus obtain the following more conceptual characterization of stability.

\begin{thm}\label{thm:stable-lim-II}
A derivator is stable if and only if homotopy finite colimits and homotopy finite limits commute.
\end{thm}
\begin{proof}
By \autoref{thm:stable-lim-I} it suffices to show that a derivator \D is stable as soon as homotopy finite colimits and homotopy finite limits commute in \D. Such a derivator is pointed and for every $X\in\D^\square$ the canonical morphism
\begin{equation}\label{eq:stable-lim-II}
C(F_2X)\toiso F(C_1X)
\end{equation}
is an isomorphism (\autoref{cor:lim-comm}). For every $x\in\D$ we consider the square
\[
X=X(x)=(i_\lrcorner)_!\pi_\lrcorner^\ast x\in\D^\square.
\]
The morphism $\pi_\lrcorner^\ast\colon\D\to\D^\lrcorner$ forms constant cospans. Since $i_\lrcorner\colon\lrcorner\to \square$ is a cosieve, $(i_\lrcorner)_!$ is left extension by zero \cite[Prop.~3.6]{groth:ptstab} and the diagram $X\in\D^\square$ looks like
\[
\xymatrix{
0\ar[r]\ar[d]&x\ar[d]^-\id\\
x\ar[r]_-\id&x.
}
\]
We calculate $CF_2(X)\cong C(\Omega x\to 0)\cong \Sigma\Omega x$ and $FC_1(X)\cong F(x\to 0)\cong x$, showing that the canonical isomorphism \eqref{eq:stable-lim-II} induces a natural isomorphism $\Sigma\Omega\toiso\id$. Using constant spans instead one also constructs a natural isomorphism $\id\toiso\Omega\Sigma$, showing that $\Sigma,\Omega\colon\D\to\D$ are equivalences. It follows from \autoref{thm:stable-known} that \D is stable.
\end{proof}

It is now straightforward to obtain the following variant of this theorem. We recall from \cite[\S9]{groth:can-can} that \textbf{left homotopy finite left Kan extensions} are left Kan extensions along functors $u\colon A\to B$ such that the slice categories $(u/b),$ $b\in B,$ admit a homotopy final functor $C_b\to(u/b)$ from a homotopy finite category~$C_b$. The point of this notion is that right exact morphisms of derivators preserve left homotopy finite left Kan extensions \cite[Thm.~9.14]{groth:can-can}.

\begin{thm}\label{thm:stable-lim-III}
The following are equivalent for a derivator \D.
\begin{enumerate}
\item The derivator \D is stable.
\item Homotopy finite colimits and homotopy finite limits commute in \D.
\item[(iii.a)] Left homotopy finite left Kan extensions and arbitrary right Kan extensions commute in \D. 
\item[(iii.b)] Arbitrary left Kan extensions and right homotopy finite right Kan extensions commute in \D.
\item[(iv.a)] Every left exact morphism $\D^A\to\D^B,A,B\in\cCat,$ preserves left homotopy finite left Kan extensions. 
\item[(iv.b)] Every right exact morphism $\D^A\to\D^B,A,B\in\cCat,$ preserves right homotopy finite right Kan extensions.
\item[(v)] The derivator \D is pointed and $C\colon\D^{[1]}\to\D$ preserves right homotopy finite right Kan extensions.
\item[(vi)] The derivator \D is pointed and $C\colon\D^{[1]}\to\D$ preserves homotopy finite limits.
\item[(vii)] The derivator \D is pointed and $C\colon\D^{[1]}\to\D$ preserves $F$.
\end{enumerate}
\end{thm}
\begin{proof}
If \D is stable, then also the shifted derivators $\D^A,A\in\cCat,$ are stable \cite[Prop.~4.3]{groth:ptstab}. Consequently, every left exact morphism $\D^A\to\D^B$ is also right exact \cite[Prop.~9.8]{groth:can-can} and it hence preserves left homotopy finite left Kan extensions \cite[Thm.~9.14]{groth:can-can}. This and a dual argument shows that statement (i) implies statements (iv.a) and (iv.b). Since right Kan extension morphisms are right adjoint morphisms and hence left exact, the implications (iv.a) implies (iii.a) and (iii.a) implies (ii) are immediate. Moreover, (ii) implies (i) by \autoref{thm:stable-lim-II}, and, by duality, it remains to incorporate the three final statements. Statement (i) implies statement (v) since $C$ is left exact in this case and it hence preserves right homotopy finite right homotopy Kan extensions \cite[Thm.~9.14]{groth:can-can}. The implications (v) implies (vi) and (vi) implies (vii) being trivial, it remains to show that (vii) implies (i) which is already taken care of by the proof of \autoref{thm:stable-lim-II}.
\end{proof}

The final characterization in \autoref{thm:stable-lim-III} is to be considered in contrast to \autoref{prop:cones-commute}. And, of course, there are various additional minor variants of the characterizations in \autoref{thm:stable-lim-III} obtained, for example, by replacing $C$ by the morphism $\cof\colon\D^{[1]}\to\D^{[1]}$. 

\begin{rmk}\label{rmk:interpretation}
A typical slogan is that spectra are obtained from pointed topological spaces if one forces the suspension to become an equivalence. This slogan is made precise by \autoref{thm:stable-known} and the fact that the derivator of spectra is the stabilization of the derivator of pointed topological spaces \cite{heller:stable}. \autoref{thm:stable-known} and \autoref{thm:stable-lim-III} make precise various additional slogans saying, for instance, that spectra are obtained from spaces or pointed spaces by forcing certain colimit and limit type constructions to commute. We illustrate this by two examples.
\begin{enumerate}
\item If one forces homotopy finite colimits and homotopy finite limits to commute in the derivator of spaces, then one obtains the derivator of spectra. 
\item If one forces partial cones and partial fibers of squares to commute in the derivator of pointed spaces, then this yields the derivator of spectra.
\end{enumerate}
\end{rmk}

\begin{rmk}\label{rmk:stable-rep-triv}
The phenomenon that certain colimits and limits commute is well-known from ordinary category theory. To mention an instance, let us recall that filtered colimits are exact in Grothendieck abelian categories, i.e., filtered colimits and finite limits commute in such categories. Additional such statements hold in locally presentable categories, Grothendieck topoi, and algebraic categories. 

Now, the phenomenon of stability is invisible to ordinary category theory; in fact, a represented derivator is stable if and only if the representing category is trivial (this follows from \autoref{thm:stable-known} since the suspension morphism is trivial in pointed represented derivators). As a consequence the commutativity statements in \autoref{thm:stable-lim-III} have no counterparts in ordinary category theory.
\end{rmk}

\section{A fun perspective on stability}
\label{sec:fun}

We conclude this paper by a characterization of pointedness and stability in terms of iterated adjoints to constant morphism morphisms.  Let us again begin with a related result in the case of pointed derivators. 

\begin{prop}\label{prop:char-ptd}
The following are equivalent for a derivator \D.
\begin{enumerate}
\item The derivator \D is pointed. 
\item The morphism $\emptyset_!\colon\D^\emptyset\to\D$ is a right adjoint.
\item The left Kan extension morphism $1_!\colon\D\to\D^{[1]}$ along the universal cosieve $1\colon\bbone\to[1]$ is a right adjoint.
\item For every cosieve $u\colon A\to B$ the left Kan extension morphism $u_!\colon\D^A\to\D^B$ is a right adjoint.
\item The morphism $\emptyset_\ast\colon\D^\emptyset\to\D$ is a left adjoint.
\item The right Kan extension morphism $0_\ast\colon\D\to\D^{[1]}$ along the universal sieve $0\colon\bbone\to[1]$ is a left adjoint.
\item For every sieve $u\colon A\to B$ the right Kan extension morphism $u_\ast\colon\D^A\to\D^B$ is a left adjoint.
\end{enumerate}
\end{prop}
\begin{proof}
By duality it suffices to show the equivalence of the first four statements. The implication (i) implies (iv) is \cite[Cor.~3.8]{groth:ptstab}. Since the empty functor $\emptyset\colon\emptyset\to\bbone$ is a cosieve it remains to show that (ii) or (iii) imply (i). The case of (ii) is taken care of by the proof of \cite[Cor.~3.5]{groth:ptstab}. In the remaining case, if $1_!$ is a right adjoint it preserves all limits and hence terminal objects. Since the terminal object in $\D([1])$ looks like $(\ast\to\ast)$, this has by \cite[Prop.~1.23]{groth:ptstab} to be isomorphic to $1_!(\ast)\cong(\emptyset\to\ast)$. Evaluating this isomorphism at $0$ shows that \D is pointed.
\end{proof}

These additional adjoint functors are sometimes referred to as \textbf{(co)exceptional inverse image functors} (see \cite[\S3]{groth:ptstab}).

\begin{rmk}\label{rmk:C-inverse}
In \cite{groth:ptstab} the cone $C\colon\D^{[1]}\to\D$ and the fiber $F\colon\D^{[1]}\to\D$ is defined in pointed derivators only, but the same formulas make perfectly well sense in arbitrary derivators. It turns out that a derivator is pointed if and only if $C$ is a left adjoint if and only if $F$ is a right adjoint. In that case, there are adjunctions $C\dashv 1_!$ and $0_\ast\dashv F$, exhibiting $C$ and $F$ as (co)exceptional inverse image functors; see \cite[Prop.~3.22]{groth:ptstab}. 
\end{rmk}

Let \D be a derivator and let $1\colon\bbone\to[1]$ classify the terminal object $1\in[1]$. In every derivator \D there are Kan extension adjunctions $(1_!,1^\ast)\colon\D\rightleftarrows\D^{[1]}$
and $(1^\ast,1_\ast)\colon\D\rightleftarrows\D^{[1]}$, and we hence have an adjoint triple
\[
1_!\dashv 1^\ast\dashv 1_\ast.
\]
Similarly, associated to the functor $0\colon\bbone\to[1]$ there is the adjoint triple
\[
0_!\dashv 0^\ast\dashv 0_\ast.
\]

\begin{prop}\label{prop:univ-sieve}
Let \D be a derivator and let $0,1\colon\bbone\to[1]$ classify the objects $0,1\in[1]$.
\begin{enumerate}
\item The morphisms $0_!,1_\ast\colon\D\to\D^{[1]}$ are fully faithful and induce an equivalence onto the full subderivator spanned by the isomorphisms. This equivalence is pseudo-natural with respect to arbitrary morphisms of derivators.
\item There are canonical isomorphism $0_!\cong \pi_{[1]}^\ast\cong 1_\ast\colon\D\to\D^{[1]}$.
\end{enumerate}
\end{prop}
\begin{proof}
Both morphisms $0_!$ and $1_\ast$ are fully faithful and the essential image consists precisely of the isomorphisms by \cite[Prop.~3.12]{groth:ptstab}. Since derivators are invariant under equivalences of prederivators, the subprederivator of isomorphisms is a derivator. The equivalence is pseudo-natural with respect to arbitrary morphisms since all morphisms preserve left Kan extensions along left adjoint functors (see \cite[Prop.~5.7]{groth:can-can} and \cite[Rmk.~6.11]{groth:can-can}). As for the second statement, there is an adjoint triple $0\dashv\pi_{[1]}\dashv 1$ and hence an induced adjoint triple $1^\ast\dashv \pi_{[1]}^\ast\dashv 0^\ast$. This yields canonical isomorphisms $1_\ast\cong\pi_{[1]}^\ast$ and $0_!\cong\pi_{[1]}^\ast$. 
\end{proof}

We refer to $\pi_{[1]}^\ast\colon\D\to\D^{[1]}$ as the \textbf{constant morphism morphism}.

\begin{cor}
In every derivator \D there is an adjoint $5$-tuple
\begin{equation}\label{eq:5tuple}
1_!\dashv 1^\ast\dashv \pi_{[1]}^\ast\dashv 0^\ast\dashv 0_\ast.
\end{equation}
\end{cor}
\begin{proof}
This is immediate from \autoref{prop:univ-sieve}.
\end{proof}

\begin{prop}
A derivator \D is pointed if and only if the adjoint $5$-tuple \eqref{eq:5tuple} extends to an adjoint $7$-tuple, which is then given by
\begin{equation}\label{eq:7tuple}
C\dashv 1_!\dashv 1^\ast\dashv \pi_{[1]}^\ast\dashv 0^\ast\dashv 0_\ast\dashv F.
\end{equation}
\end{prop}
\begin{proof}
This is immediate from \autoref{prop:char-ptd} and \autoref{rmk:C-inverse}.
\end{proof}

\begin{rmk}
While in any pointed derivator there is by \cite[Prop.~3.20]{groth:ptstab} an adjunction
\[
(\cof,\fib)\colon\D^{[1]}\rightleftarrows\D^{[1]},
\]
in pointed derivators the morphism $C$ is the sixth left adjoint of $F$. 
\end{rmk}

\begin{thm}\label{thm:stable-fun}
The following are equivalent for a pointed derivator.
\begin{enumerate}
\item The derivator is stable.
\item The adjoint $7$-tuple \eqref{eq:7tuple} extends to a doubly-infinite chain of adjoint morphisms.
\end{enumerate}
\end{thm}
\begin{proof}
Let \D be a pointed derivator such that \eqref{eq:7tuple} extends to such an infinite chain. The morphism $C\colon\D^{[1]}\to\D$ is then a right adjoint morphism of pointed derivators. As a left exact morphism of pointed derivators, $C$ preserves fibers \cite[Prop.~9.6]{groth:can-can}, and \autoref{thm:stable-lim-III} implies that \D is stable.

Conversely, let \D be a stable derivator. By \autoref{prop:stable-known-mod} there are natural isomorphisms
\[
\Sigma F\toiso C\qquad\text{and}\qquad F\toiso\Omega C.
\]
Since $\Sigma$ and $\Omega$ are equivalences in stable derivators (\autoref{thm:stable-known}), this shows that the outer morphisms in the adjoint $7$-tuple \eqref{eq:7tuple} match up to an equivalence. This implies that the adjoint $7$-tuple can be extended to a doubly-infinite chain of adjoint morphisms and that this chain has period six (in the obvious sense).
\end{proof}

We conclude by offering a first interpretation and visualization of this chain of morphisms.

\begin{rmk}
Let \D be a stable derivator. Then a few additional adjoint morphisms in the doubly-infinite sequence extending \eqref{eq:7tuple} are given by:
\[
\ldots\dashv\pi^\ast\Omega\dashv \Sigma 0^\ast\dashv 0_\ast\Omega\dashv C\dashv 1_!\dashv 1^\ast\dashv \pi^\ast\dashv 0^\ast\dashv 0_\ast\dashv F\dashv 1_!\Sigma\dashv \Omega 1^\ast\dashv \pi^\ast\Sigma\dashv\ldots
\]
In fact, this is immediate from the proof of \autoref{thm:stable-fun}.

In order to not get lost in all these morphisms, let us recall that Barratt--Puppe sequences in a stable derivator \D can be thought of as refinements of the more classical distinguished triangles. More precisely, associated to $(f\colon x\to y)\in\D^{[1]}$ there is the Barratt--Puppe sequence $BP(f)$ generated by $f$. This is a coherent diagram as in \autoref{fig:BP-sequence-stable} which vanishes on the boundary stripes and which makes all squares bicartesian.
\begin{figure}
\begin{equation}
\vcenter{
\xymatrix@-1pc{
\ar@{}[dr]|{\ddots}&\ar@{}[dr]|{\ddots}&\ar@{}[dr]|{\ddots}&&&&&&\\
\ar@{}[dr]|{\ddots}&\Omega Ff\ar[r]\ar[d]\pullbackcorner&\Omega x\ar[r]\ar[d]\pullbackcorner&0\ar[d]&&&&&\\
&0\ar[r]&\Omega y\ar[r]\ar[d]\pushoutcorner\pullbackcorner&Ff\ar[r]\ar[d]\pushoutcorner\pullbackcorner&0\ar[d]&&&&\\
&&0\ar[r]&x\ar[r]^-f\ar[d]\pushoutcorner\pullbackcorner&y\ar[r]\ar[d]\pushoutcorner\pullbackcorner&0\ar[d]&&&\\
&&&0\ar[r]&Cf\ar[r]\ar[d]\pushoutcorner\pullbackcorner&\Sigma x\ar[r]\ar[d]\pushoutcorner\pullbackcorner&0\ar[d]\ar@{}[rd]|{\ddots}&&\\
&&&&0\ar[r]\ar@{}[dr]|{\ddots}&\Sigma y\ar[r]\ar@{}[rd]|{\ddots}\pushoutcorner&\Sigma Cf\ar@{}[dr]|{\ddots}\pushoutcorner&\\
&&&&&&&&
}
}
\end{equation}
\caption{The Barratt--Puppe sequence of $f$}
\label{fig:BP-sequence-stable}
\end{figure}
(It turns out that $BP$ defines an equivalence of derivators (see \cite[Thm.~4.5]{gst:Dynkin-A}).)

Now, one half of the morphisms in the doubly-infinite chain simply amount to traveling in the Barratt--Puppe sequence in \autoref{fig:BP-sequence-stable}. If we imagine to sit on the morphism $f$ in $BP(f)$, then for every $n\geq 1$ an application of the $(2n\text{-}1)$-th left adjoint of $\pi^\ast$ to $f$ amounts to traveling $n$ steps in the positive direction. For low values this yields $y,Cf,\Sigma x,\Sigma y,$ and so on. There is a similar interpretation of the iterated right adjoints to $\pi^\ast$.

In order to obtain a similar visualization of the remaining adjoints, let us consider the Barratt--Puppe sequence $BP(\pi_{[1]}^\ast x), x\in\D$, of a constant morphism which then looks like \autoref{fig:BP-constant}.
\begin{figure}
\begin{equation}
\vcenter{
\xymatrix@-1pc{
\ar@{}[dr]|{\ddots}&\ar@{}[dr]|{\ddots}&\ar@{}[dr]|{\ddots}&&&&&&\\
\ar@{}[dr]|{\ddots}&0\ar[r]\ar[d]\pullbackcorner&\Omega x\ar[r]\ar[d]\pullbackcorner&0\ar[d]&&&&&\\
&0\ar[r]&\Omega x\ar[r]\ar[d]\pushoutcorner\pullbackcorner&0\ar[r]\ar[d]\pushoutcorner\pullbackcorner&0\ar[d]&&&&\\
&&0\ar[r]&x\ar[r]^-\id\ar[d]\pushoutcorner\pullbackcorner&x\ar[r]\ar[d]\pushoutcorner\pullbackcorner&0\ar[d]&&&\\
&&&0\ar[r]&0\ar[r]\ar[d]\pushoutcorner\pullbackcorner&\Sigma x\ar[r]\ar[d]\pushoutcorner\pullbackcorner&0\ar[d]\ar@{}[rd]|{\ddots}&&\\
&&&&0\ar[r]\ar@{}[dr]|{\ddots}&\Sigma x\ar[r]\ar@{}[rd]|{\ddots}\pushoutcorner&0\ar@{}[dr]|{\ddots}\pushoutcorner&\\
&&&&&&&&
}
}
\end{equation}
\caption{The Barratt--Puppe sequence of $\pi_{[1]}^\ast x$}
\label{fig:BP-constant}
\end{figure}
While $\pi^\ast$ points at the constant morphism in the middle of \autoref{fig:BP-constant}, for every $n$ the remaining $2n$-th adjoints to $\pi^\ast$ classify suitable iterated rotations of this morphism.
\end{rmk}

\bibliographystyle{alpha}
\bibliography{stability}

\end{document}

%% file: decls.tex
\usepackage{amssymb,amsmath,stmaryrd,mathrsfs}
\def\definetac{\newif\iftac}    
\ifx\tactrue\undefined
  \definetac
  \ifx\state\undefined\tacfalse\else\tactrue\fi
\fi
\iftac\else\usepackage{amsthm}\fi
\usepackage[all,2cell]{xy}
\UseAllTwocells
\usepackage{enumitem}
\usepackage{xcolor}
\definecolor{darkgreen}{rgb}{0,0.45,0} 
\usepackage[pagebackref,colorlinks,citecolor=darkgreen,linkcolor=darkgreen]{hyperref}
\usepackage{mathtools}          
\usepackage{tikz}
\usepackage{braket}             

\usepackage{url}                
\usepackage{xspace}             

\makeatletter
\let\ea\expandafter

\def\mdef#1#2{\ea\ea\ea\gdef\ea\ea\noexpand#1\ea{\ea\ensuremath\ea{#2}\xspace}}
\def\alwaysmath#1{\ea\ea\ea\global\ea\ea\ea\let\ea\ea\csname your@#1\endcsname\csname #1\endcsname
  \ea\def\csname #1\endcsname{\ensuremath{\csname your@#1\endcsname}\xspace}}

\DeclareRobustCommand\widecheck[1]{{\mathpalette\@widecheck{#1}}}
\def\@widecheck#1#2{%
    \setbox\z@\hbox{\m@th$#1#2$}%
    \setbox\tw@\hbox{\m@th$#1%
       \widehat{%
          \vrule\@width\z@\@height\ht\z@
          \vrule\@height\z@\@width\wd\z@}$}%
    \dp\tw@-\ht\z@
    \@tempdima\ht\z@ \advance\@tempdima2\ht\tw@ \divide\@tempdima\thr@@
    \setbox\tw@\hbox{%
       \raise\@tempdima\hbox{\scalebox{1}[-1]{\lower\@tempdima\box
\tw@}}}%
    {\ooalign{\box\tw@ \cr \box\z@}}}

\newcount\foreachcount

\def\foreachletter#1#2#3{\foreachcount=#1
  \ea\loop\ea\ea\ea#3\@alph\foreachcount
  \advance\foreachcount by 1
  \ifnum\foreachcount<#2\repeat}

\def\foreachLetter#1#2#3{\foreachcount=#1
  \ea\loop\ea\ea\ea#3\@Alph\foreachcount
  \advance\foreachcount by 1
  \ifnum\foreachcount<#2\repeat}

\def\definescr#1{\ea\gdef\csname s#1\endcsname{\ensuremath{\mathscr{#1}}\xspace}}
\foreachLetter{1}{27}{\definescr}
\def\definecal#1{\ea\gdef\csname c#1\endcsname{\ensuremath{\mathcal{#1}}\xspace}}
\foreachLetter{1}{27}{\definecal}
\def\definebold#1{\ea\gdef\csname b#1\endcsname{\ensuremath{\mathbf{#1}}\xspace}}
\foreachLetter{1}{27}{\definebold}
\def\definebb#1{\ea\gdef\csname l#1\endcsname{\ensuremath{\mathbb{#1}}\xspace}}
\foreachLetter{1}{27}{\definebb}
\def\definefrak#1{\ea\gdef\csname f#1\endcsname{\ensuremath{\mathfrak{#1}}\xspace}}
\foreachletter{1}{9}{\definefrak} 
\foreachletter{10}{27}{\definefrak}
\def\definebar#1{\ea\gdef\csname #1bar\endcsname{\ensuremath{\overline{#1}}\xspace}}
\foreachLetter{1}{27}{\definebar}
\foreachletter{1}{8}{\definebar} 
\foreachletter{9}{15}{\definebar} 
\foreachletter{16}{27}{\definebar}
\def\definetil#1{\ea\gdef\csname #1til\endcsname{\ensuremath{\widetilde{#1}}\xspace}}
\foreachLetter{1}{27}{\definetil}
\foreachletter{1}{27}{\definetil}
\def\definehat#1{\ea\gdef\csname #1hat\endcsname{\ensuremath{\widehat{#1}}\xspace}}
\foreachLetter{1}{27}{\definehat}
\foreachletter{1}{27}{\definehat}
\def\definechk#1{\ea\gdef\csname #1chk\endcsname{\ensuremath{\widecheck{#1}}\xspace}}
\foreachLetter{1}{27}{\definechk}
\foreachletter{1}{27}{\definechk}
\def\defineul#1{\ea\gdef\csname u#1\endcsname{\ensuremath{\underline{#1}}\xspace}}
\foreachLetter{1}{27}{\defineul}
\foreachletter{1}{27}{\defineul}

\def\autofmt@n#1\autofmt@end{\mathrm{#1}}
\def\autofmt@b#1\autofmt@end{\mathbf{#1}}
\def\autofmt@l#1#2\autofmt@end{\mathbb{#1}\mathsf{#2}}
\def\autofmt@c#1#2\autofmt@end{\mathcal{#1}\mathit{#2}}
\def\autofmt@s#1#2\autofmt@end{\mathscr{#1}\mathit{#2}}
\def\autofmt@f#1\autofmt@end{\mathsf{#1}}
\def\autofmt@u#1\autofmt@end{\underline{\smash{\mathsf{#1}}}}
\def\autofmt@U#1\autofmt@end{\underline{\underline{\smash{\mathsf{#1}}}}}
\def\autofmt@h#1\autofmt@end{\widehat{#1}}
\def\autofmt@r#1\autofmt@end{\overline{#1}}
\def\autofmt@t#1\autofmt@end{\widetilde{#1}}
\def\autofmt@k#1\autofmt@end{\check{#1}}

\def\auto@drop#1{}
\def\autodef#1{\ea\ea\ea\@autodef\ea\ea\ea#1\ea\auto@drop\string#1\autodef@end}
\def\@autodef#1#2#3\autodef@end{%
  \ea\def\ea#1\ea{\ea\ensuremath\ea{\csname autofmt@#2\endcsname#3\autofmt@end}\xspace}}
\def\autodefs@end{blarg!}
\def\autodefs#1{\@autodefs#1\autodefs@end}
\def\@autodefs#1{\ifx#1\autodefs@end%
  \def\autodefs@next{}%
  \else%
  \def\autodefs@next{\autodef#1\@autodefs}%
  \fi\autodefs@next}

\DeclareSymbolFont{bbold}{U}{bbold}{m}{n}
\DeclareSymbolFontAlphabet{\mathbbb}{bbold}

\newcommand{\bbone}{\ensuremath{\mathbbb{1}}\xspace}

\mdef\delbar{\overline{\partial}}

\mdef\hf{\textstyle\frac12 }
\mdef\thrd{\textstyle\frac13 }
\mdef\qtr{\textstyle\frac14 }

\SelectTips{cm}{}
\newdir{ >}{{}*!/-10pt/@{>}}    
\newcommand{\pushoutcorner}[1][dr]{\save*!/#1+1.2pc/#1:(1,-1)@^{|-}\restore}
\newcommand{\pullbackcorner}[1][dr]{\save*!/#1-1.2pc/#1:(-1,1)@^{|-}\restore}

\mdef\Id{\mathrm{Id}}
\mdef\id{\mathrm{id}}
\alwaysmath{ell}
\alwaysmath{infty}
\alwaysmath{odot}
\def\frc#1/#2.{\frac{#1}{#2}}   
\mdef\ten{\mathrel{\otimes}}

\mdef\sqten{\mathrel{\boxtimes}}

\DeclareRobustCommand\widecheck[1]{{\mathpalette\@widecheck{#1}}}
\def\@widecheck#1#2{%
    \setbox\z@\hbox{\m@th$#1#2$}%
    \setbox\tw@\hbox{\m@th$#1%
       \widehat{%
          \vrule\@width\z@\@height\ht\z@
          \vrule\@height\z@\@width\wd\z@}$}%
    \dp\tw@-\ht\z@
    \@tempdima\ht\z@ \advance\@tempdima2\ht\tw@ \divide\@tempdima\thr@@
    \setbox\tw@\hbox{%
       \raise\@tempdima\hbox{\scalebox{1}[-1]{\lower\@tempdima\box
\tw@}}}%
    {\ooalign{\box\tw@ \cr \box\z@}}}

\DeclareMathOperator\colim{colim}

\mdef\we{\overset{\sim}{\longrightarrow}}
\mdef\leftwe{\overset{\sim}{\longleftarrow}}

\let\xto\xrightarrow

\def\rightarrowtailfill@{\arrowfill@{\Yright\joinrel\relbar}\relbar\rightarrow}
\newcommand\xrightarrowtail[2][]{\ext@arrow 0055{\rightarrowtailfill@}{#1}{#2}}

\def\twoheadrightarrowfill@{\arrowfill@{\relbar\joinrel\relbar}\relbar\twoheadrightarrow}
\newcommand\xtwoheadrightarrow[2][]{\ext@arrow 0055{\twoheadrightarrowfill@}{#1}{#2}}

\def\slashedarrowfill@#1#2#3#4#5{%
  $\m@th\thickmuskip0mu\medmuskip\thickmuskip\thinmuskip\thickmuskip
   \relax#5#1\mkern-7mu%
   \cleaders\hbox{$#5\mkern-2mu#2\mkern-2mu$}\hfill
   \mathclap{#3}\mathclap{#2}%
   \cleaders\hbox{$#5\mkern-2mu#2\mkern-2mu$}\hfill
   \mkern-7mu#4$%
}
\def\rightslashedarrowfill@{%
  \slashedarrowfill@\relbar\relbar\mapstochar\rightarrow}
\newcommand\xslashedrightarrow[2][]{%
  \ext@arrow 0055{\rightslashedarrowfill@}{#1}{#2}}
\mdef\hto{\xslashedrightarrow{}}
\mdef\htoo{\xslashedrightarrow{\quad}}

\def\toiso{\xto{\smash{\raisebox{-.5mm}{$\scriptstyle\sim$}}}}

\long\def\my@drawfill#1#2;{%
\@skipfalse
\fill[#1,draw=none] #2;
\@skiptrue
\draw[#1,fill=none] #2;
}
\newif\if@skip
\newcommand{\skipit}[1]{\if@skip\else#1\fi}
\newcommand{\drawfill}[1][]{\my@drawfill{#1}}

\newif\ifhyperref
\@ifpackageloaded{hyperref}{\hyperreftrue}{\hyperreffalse}
\iftac
  \let\your@state\state
  \def\state#1{\gdef\currthmtype{#1}\your@state{#1}}
  \let\your@staterm\staterm
  \def\staterm#1{\gdef\currthmtype{#1}\your@staterm{#1}}
  \let\defthm\newtheorem
  \def\currthmtype{}
  \ifhyperref
    \def\autoref#1{\ref*{label@name@#1}~\ref{#1}}
  \else
    \def\autoref#1{\ref{label@name@#1}~\ref{#1}}
  \fi
  \AtBeginDocument{%
    \let\old@label\label%
    \def\label#1{%
      {\let\your@currentlabel\@currentlabel%
        \edef\@currentlabel{\currthmtype}%
        \old@label{label@name@#1}}%
      \old@label{#1}}
  }
\else
  \ifhyperref
    \def\defthm#1#2{%
      \newtheorem{#1}{#2}[section]%
      \expandafter\def\csname #1autorefname\endcsname{#2}%
      \expandafter\let\csname c@#1\endcsname\c@thm}
  \else
    \def\defthm#1#2{\newtheorem{#1}[thm]{#2}}
    \ifx\SK@label\undefined\let\SK@label\label\fi
    \let\old@label\label
    \let\your@thm\@thm
    \def\@thm#1#2#3{\gdef\currthmtype{#3}\your@thm{#1}{#2}{#3}}
    \def\currthmtype{}
    \def\label#1{{\let\your@currentlabel\@currentlabel\def\@currentlabel%
        {\currthmtype~\your@currentlabel}%
        \SK@label{#1@}}\old@label{#1}}
    \def\autoref#1{\ref{#1@}}
  \fi
\fi

\newtheorem{thm}{Theorem}[section]

\defthm{cor}{Corollary}
\defthm{prop}{Proposition}
\defthm{lem}{Lemma}
\defthm{sch}{Scholium}
\defthm{assume}{Assumption}
\defthm{claim}{Claim}
\defthm{conj}{Conjecture}
\defthm{hyp}{Hypothesis}
\defthm{fact}{Fact}
\iftac\theoremstyle{plain}\else\theoremstyle{definition}\fi
\defthm{defn}{Definition}
\defthm{notn}{Notation}
\iftac\theoremstyle{plain}\else\theoremstyle{remark}\fi
\defthm{rmk}{Remark}
\defthm{eg}{Example}
\defthm{egs}{Examples}
\defthm{ex}{Exercise}
\defthm{ceg}{Counterexample}
\defthm{con}{Construction}
\defthm{warn}{Warning}

\def\thmqedhere{\expandafter\csname\csname @currenvir\endcsname @qed\endcsname}

\setitemize[1]{leftmargin=2em}
\setenumerate[1]{leftmargin=*}

\iftac
  \let\c@equation\c@subsection
\else
  \let\c@equation\c@thm
\fi
\numberwithin{equation}{section}

\@ifpackageloaded{mathtools}{\mathtoolsset{showonlyrefs,showmanualtags}}{}

\alwaysmath{alpha}
\alwaysmath{beta}
\alwaysmath{gamma}
\alwaysmath{Gamma}
\alwaysmath{delta}
\alwaysmath{Delta}
\alwaysmath{epsilon}
\mdef\ep{\varepsilon}
\alwaysmath{zeta}
\alwaysmath{eta}
\alwaysmath{theta}
\alwaysmath{Theta}
\alwaysmath{iota}
\alwaysmath{kappa}
\alwaysmath{lambda}
\alwaysmath{Lambda}
\alwaysmath{mu}
\alwaysmath{nu}
\alwaysmath{xi}
\alwaysmath{pi}
\alwaysmath{rho}
\alwaysmath{sigma}
\alwaysmath{Sigma}
\alwaysmath{tau}
\alwaysmath{upsilon}
\alwaysmath{Upsilon}
\alwaysmath{phi}
\alwaysmath{Pi}
\alwaysmath{Phi}
\mdef\ph{\varphi}
\alwaysmath{chi}
\alwaysmath{psi}
\alwaysmath{Psi}
\alwaysmath{omega}
\alwaysmath{Omega}

\makeatother


%% file: stability.bbl
\def\cprime{$'$}
\begin{thebibliography}{GGN13}

\bibitem[Cis03]{cisinski:direct}
Denis-Charles Cisinski.
\newblock Images directes cohomologiques dans les cat\'egories de mod\`eles.
\newblock {\em Ann. Math. Blaise Pascal}, 10(2):195--244, 2003.

\bibitem[Cis04]{cisinski:loc-min}
Denis-Charles Cisinski.
\newblock Le localisateur fondamental minimal.
\newblock {\em Cah. Topol. G\'eom. Diff\'er. Cat\'eg.}, 45(2):109--140, 2004.

\bibitem[Cis08]{cisinski:derived-kan}
Denis-Charles Cisinski.
\newblock Propri\'et\'es universelles et extensions de {K}an d\'eriv\'ees.
\newblock {\em Theory Appl. Categ.}, 20:No. 17, 605--649, 2008.

\bibitem[Fra96]{franke:adams}
Jens Franke.
\newblock Uniqueness theorems for certain triangulated categories with an
  {A}dams spectral sequence, 1996.
\newblock Preprint.

\bibitem[GGN13]{ggn:infinite}
David Gepner, Moritz Groth, and Thomas Nikolaus.
\newblock Universality of multiplicative infinite delooping machines.
\newblock \url{http://arxiv.org/pdf/1305.4550}, 2013.
\newblock preprint, to appear in Algebraic \& Geometric Topology.

\bibitem[Goo92]{goodwillie:II}
Thomas~G. Goodwillie.
\newblock Calculus. {II}. {A}nalytic functors.
\newblock {\em $K$-Theory}, 5(4):295--332, 1991/92.

\bibitem[GPS14]{gps:mayer}
Moritz Groth, Kate Ponto, and Michael Shulman.
\newblock {M}ayer--{V}ietoris sequences in stable derivators.
\newblock {\em Homology, Homotopy Appl.}, 16:265--294, 2014.

\bibitem[Gro]{grothendieck:derivators}
Alexander Grothendieck.
\newblock Les d\'erivateurs.
\newblock \\\url{http://www.math.jussieu.fr/~maltsin/groth/Derivateurs.html}.
\newblock Manuscript.

\bibitem[Gro13]{groth:ptstab}
Moritz Groth.
\newblock Derivators, pointed derivators, and stable derivators.
\newblock {\em Algebr. Geom. Topol.}, 13:313--374, 2013.

\bibitem[Gro16a]{groth:intro-to-der-1}
Moritz Groth.
\newblock Book project on derivators, volume {I}, 2016.
\newblock Book project in preparation, draft available at
  \url{http://www.math.uni-bonn.de/people/mgroth/monos.htmpl}.

\bibitem[Gro16b]{groth:can-can}
Moritz Groth.
\newblock Revisiting the canonicity of canonical triangulations.
\newblock \url{http://arxiv.org/abs/1602.04846}, 2016.

\bibitem[Gro16c]{groth:formal}
Moritz Groth.
\newblock Universal formulas in abstract homotopy theories.
\newblock In preparation, 2016.

\bibitem[G{\v S}14a]{gst:Dynkin-A}
Moritz Groth and Jan {\v S}{\v t}ov{\'\i}{\v c}ek.
\newblock Abstract representation theory of {D}ynkin quivers of type~{A}.
\newblock To appear in {\em Advances in Mathematics}. Available at
  arXiv:1409.5003, 2014.

\bibitem[G{\v S}14b]{gst:basic}
Moritz Groth and Jan {\v S}{\v t}ov{\'\i}{\v c}ek.
\newblock Tilting theory via stable homotopy theory.
\newblock To appear in {\em Crelle's Journal}. Available at arXiv:1401.6451,
  2014.

\bibitem[G{\v S}15]{gst:acyclic}
Moritz Groth and Jan {\v S}{\v t}ov{\'\i}{\v c}ek.
\newblock Abstract tilting theory for quivers and related categories.
\newblock arXiv:1512.06267, 2015.

\bibitem[GS16a]{gs:enriched}
Moritz Groth and Michael Shulman.
\newblock Something on enriched derivators and weighted homotopy colimits.
\newblock In preparation, 2016.

\bibitem[G{\v S}16b]{gst:acyclic-Serre}
Moritz Groth and Jan {\v S}{\v t}ov{\'\i}{\v c}ek.
\newblock Spectral {S}erre duality of acyclic quivers.
\newblock In preparation, 2016.

\bibitem[G{\v S}16c]{gst:tree}
Moritz Groth and Jan {\v S}{\v t}ov{\'\i}{\v c}ek.
\newblock Tilting theory for trees via stable homotopy theory.
\newblock {\em Journal of Pure and Applied Algebra}, 220(6):2324 -- 2363, 2016.

\bibitem[Hel88]{heller:htpythies}
Alex Heller.
\newblock Homotopy theories.
\newblock {\em Mem. Amer. Math. Soc.}, 71(383):vi+78, 1988.

\bibitem[Hel97]{heller:stable}
Alex Heller.
\newblock Stable homotopy theories and stabilization.
\newblock {\em J. Pure Appl. Algebra}, 115(2):113--130, 1997.

\bibitem[Mal01]{maltsiniotis:seminar}
Georges Maltsiniotis.
\newblock Groupe de travail sur les d\'erivateurs.
\newblock \url{http://people.math.jussieu.fr/~maltsin/textes.html}, 2001.
\newblock Seminar in Paris.

\bibitem[Mal07]{maltsiniotis:k-theory}
Georges Maltsiniotis.
\newblock La {$K$}-th\'eorie d'un d\'erivateur triangul\'e.
\newblock In {\em Categories in algebra, geometry and mathematical physics},
  volume 431 of {\em Contemp. Math.}, pages 341--368. Amer. Math. Soc.,
  Providence, RI, 2007.

\bibitem[Mal12]{maltsiniotis:htpy-exact}
Georges Maltsiniotis.
\newblock Carr\'es exacts homotopiques et d\'erivateurs.
\newblock {\em Cah. Topol. G\'eom. Diff\'er. Cat\'eg.}, 53(1):3--63, 2012.

\bibitem[PS14]{ps:linearity}
Kate Ponto and Mike Shulman.
\newblock The linearity of traces in monoidal categories and bicategories.
\newblock arXiv:1406.7854, 2014.

\end{thebibliography}
